\documentclass[12pt,a4paper]{amsart}

\usepackage[T1]{fontenc}
\usepackage{lmodern}
\usepackage[headings]{fullpage}

\usepackage{amssymb}
\usepackage{amsthm}

\theoremstyle{plain}

\newtheorem{theo}{Theorem}[section]
\newtheorem{coro}[theo]{Corollary}

\newtheorem{prop}[theo]{Proposition}
\newtheorem{rema}[theo]{Remark}
\newtheorem*{theoA}{Theorem A}
\newtheorem*{theoB}{Theorem B}

\newcommand{\Qp}{\mathbf{Q}_p}
\newcommand{\ZZ}{\mathbf{Z}}
\newcommand{\QQ}{\mathbf{Q}}
\newcommand{\BB}{\mathbf{B}}
\newcommand{\OO}{\mathcal{O}}
\newcommand{\Qpbar}{\overline{\QQ}_p}
\newcommand{\Qpnrhat}{\widehat{\QQ}_p^{\mathrm{nr}}}
\newcommand{\bcris}{\mathbf{B}_{\mathrm{cris}}} 
\newcommand{\be}{\mathbf{B}_{\mathrm{e}}} 
\newcommand{\bee}{\mathbf{B}_{\mathrm{e},E}} 
\newcommand{\bdr}{\mathbf{B}_{\mathrm{dR}}}  
  
\newcommand{\Hom}{\operatorname{Hom}}
\newcommand{\Emb}{\Sigma}
\newcommand{\End}{\operatorname{End}}
\newcommand{\Ext}{\operatorname{Ext}}
\newcommand{\Tr}{\operatorname{Tr}}
\newcommand{\Nm}{\operatorname{N}}

\newcommand{\GL}{\operatorname{GL}}
\newcommand{\Gal}{\operatorname{Gal}}
\newcommand{\Frac}{\operatorname{Frac}}
\newcommand{\dr}{\mathrm{dR}}
\newcommand{\Id}{\operatorname{Id}}
\newcommand{\nbf}{\mathbf{n}}

\renewcommand{\phi}{\varphi}
\renewcommand{\geq}{\geqslant}
\renewcommand{\leq}{\leqslant} 

\title[Triangulable tensor products of $B$-pairs]{On triangulable tensor products of $B$-pairs and trianguline representations}

\author{Laurent Berger}
\address{Laurent Berger \\ 
UMPA de l'ENS de Lyon \\
UMR 5669 du CNRS}
\email{laurent.berger@ens-lyon.fr}
\urladdr{perso.ens-lyon.fr/laurent.berger/}

\author{Giovanni Di Matteo}
\address{Giovanni Di Matteo}
\email{math@dimatteo.is}
\urladdr{dimatteo.is/Mathematics}

\date{\today}

\begin{document}

\begin{abstract}
We show that if $V$ and $V'$ are two $p$-adic representations of $\operatorname{Gal}(\overline{\mathbf{Q}}_p/\mathbf{Q}_p)$ whose tensor product is trianguline, then $V$ and $V'$ are both potentially trianguline.
\end{abstract}

\subjclass[2010]{11F80 (11F33; 11S15; 11S20; 11S31; 14F30)}

\keywords{Trianguline representation; $B$-pair; $p$-adic Hodge theory; $\mathbf{B}_e$-representation}

\maketitle

\tableofcontents

\setlength{\baselineskip}{18pt}

\section*{Introduction}\label{intro}

The notion of a trianguline representation of $G_{\Qp} = \Gal(\Qpbar/\Qp)$ was introduced by Colmez \cite{PC08} in the context of his work on the $p$-adic local Langlands correspondence for $\GL_2(\Qp)$. Examples of trianguline representations include the semi-stable representations of $G_{\Qp}$ as well as the $p$-adic representations of $G_{\Qp}$ attached to overconvergent cuspidal eigenforms of finite slope (theorem 6.3 of \cite{K03} and proposition 4.3 of \cite{PC08}). The category of all trianguline representations of $G_{\Qp}$ is stable under extensions, tensor products, and duals. We refer the reader to the book \cite{BC09} and the survey \cite{LB11} for a detailed discussion of trianguline representations. Let us at least mention the following analogue of the Fontaine-Mazur conjecture: if $V$ is an irreducible $2$-dimensional $p$-adic representation of $\Gal(\overline{\QQ}/\QQ)$ that is unramified at $\ell$ for almost all $\ell \neq p$, and whose restriction to a decomposition group at $p$ is trianguline, then $V$ is a twist of the Galois representation attached to an overconvergent cuspidal eigenform of finite slope. This conjecture is a theorem of Emerton (\S 1.2.2 of \cite{ME11}) under additional technical hypothesis on $V$. The trianguline property is in general a condition at $p$ reflecting (conjecturally at least) the fact that a $p$-adic representation comes from a $p$-adic automorphic form. This theme is pursued, for example, in \cite{DH17}, \cite{JB17} and \cite{AC17}.

If $K$ is a finite extension of $\Qp$, we also have the notion of a trianguline representation of $G_K = \Gal(\Qpbar/K)$. We say that a representation $V$ of $G_K$ is potentially trianguline if there exists a finite extension $L/K$ such that the restriction of $V$ to $G_L$ is trianguline. The goal of this article is to prove the following theorem.

\begin{theoA}
\label{theoA}
If $V$ and $V'$ are two non-zero $p$-adic representations of $G_{\Qp}$ whose tensor product is trianguline, then $V$ and $V'$ are both potentially trianguline.
\end{theoA}

We now give more details about the contents of this article. The definition of ``trianguline''  can be given either in terms of $(\phi,\Gamma)$-modules over the Robba ring, or in terms of $B$-pairs. In this article, we use the theory of $B$-pairs, which was introduced in \cite{LB08}. We remark in passing that $B$-pairs are the same as $G_K$-equivariant bundles on the Fargues-Fontaine curve \cite{FF18}. Let $K$ be a finite extension of $\Qp$. Let $\bdr^+$, $\bdr$ and $\be = (\bcris)^{\phi=1}$ be some of Fontaine's rings of $p$-adic periods \cite{F94per}. A $B$-pair is a pair $W=(W_e,W_{\dr}^+)$ where $W_e$ is a free $\be$-module of finite rank endowed with a continuous semi-linear action of $G_K$, and $W_{\dr}^+$ is a $G_K$-stable $\bdr^+$-lattice in $W_{\dr} = \bdr \otimes_{\be} W_e$. If $V$ is a $p$-adic representation of $G_K$, then $W(V) = (\be \otimes_{\Qp} V, \bdr^+ \otimes_{\Qp} V)$ is a $B$-pair. If $E$ is a finite extension of $\Qp$, the definition of $B$-pairs can be extended to $E$-linear objects, and we get objects called $\BB^{\otimes E}_{\mid K}$-pairs in \cite{BC10} or $E$-$B$-pairs of $G_K$ in \cite{KN09}. They are pairs $W=(W_e,W_{\dr}^+)$ where $W_e$ is a free $E \otimes_{\Qp} \be$-module of finite rank endowed with a continuous semi-linear action of $G_K$, and $W_{\dr}^+$ is a $G_K$-stable $E \otimes_{\Qp} \bdr^+$-lattice in $W_{\dr} = (E \otimes_{\Qp} \bdr) \otimes_{E \otimes_{\Qp} \be} W_e$. Note that the action of $G_K$ is $E$-linear.

We say (definition 1.15 of \cite{KN09}) that a $\BB^{\otimes E}_{\mid K}$-pair $W$ is split triangulable if $W$ is a successive extension of objects of rank $1$, triangulable if there exists a finite extension $F/E$ such that the $\BB^{\otimes F}_{\mid K}$-pair $F \otimes_E W$ is split triangulable, and potentially triangulable if there exists a finite extension $L/K$ such that the $\BB^{\otimes E}_{\mid L}$-pair $W |_{G_L}$ is triangulable. If $V$ is a $p$-adic representation of $G_K$, we say that $V$ is trianguline if $W(V)$ is triangulable. 

Let $\Delta$ be a set of rank $1$ semi-linear $E \otimes_{\Qp} \be$-representations of $G_K$. We say that a $\BB^{\otimes E}_{\mid K}$-pair is split $\Delta$-triangulable if it is split triangulable, and the rank $1$ $E \otimes_{\Qp} \be$-representations of $G_K$ that come from the triangulation are all in $\Delta$. Let $\Delta(\Qp)$ be the set of rank $1$ $E \otimes_{\Qp} \be$-representations of $G_K$ that extend to $G_{\Qp}$. Theorem A then results from the following more general result (theorem \ref{maintheo}), applied to $K=\Qp$.

\begin{theoB}
\label{theoB}
If $X$ and $Y$ are two non-zero $\BB^{\otimes E}_{\mid K}$-pairs whose tensor product is $\Delta(\Qp)$-triangulable, then $X$ and $Y$ are both potentially triangulable.
\end{theoB}

The proof of theorem B relies on the study of $E \otimes_{\Qp} \be$-representations of $G_K$ as well as on the study of the slopes, weights and cohomology of $\BB^{\otimes E}_{\mid K}$-pairs. The ring $E \otimes_{\Qp} \be$ has many non-trivial units, which makes the study of $\BB^{\otimes E}_{\mid K}$-pairs more difficult than when $E=\Qp$. Note finally that some of the results of this article already appear in \cite{DMphd}.

\section{Reminders and complements}

If $K$ is a finite extension of $\Qp$, let $G_K = \Gal(\overline{\mathbf{Q}}_p/K)$. Let $E$ be a finite Galois extension of $\Qp$ such that $K \subset E$, and let $\Emb=\Gal(E/\Qp)$. Let $E_0$ be the maximal unramified extension of $\Qp$ inside $E$. Let $\bdr^+$, $\bdr$, $\bcris^+$ and $\bcris$ be Fontaine's rings of $p$-adic periods (see for instance \cite{F94per}). They are all equipped with an action of $G_{\Qp}$, and $\bcris^+$ and $\bcris$ have in addition a Frobenius map $\phi$. Let $\be = (\bcris)^{\phi=1}$ and $\bee = E \otimes_{\Qp} \be$. The group $G_{\Qp}$ acts $E$-linearly on $\bee$.

\begin{prop}
\label{bepid}
The ring $\bee$ is a principal ideal domain.
\end{prop}

\begin{proof}
The ring $\bee$ is a B\'ezout domain; for $E=\Qp$ this is shown in proposition 1.1.9 of \cite{LB08}, and the same argument is used to show the general case in lemma 1.6 of \cite{KN09}. By theorem 6.5.2 of \cite{FF18}, the ring $\be$ is a principal ideal domain, and therefore $\bee$ is a principal ideal domain as well, since it is a quotient of the polynomial ring $\be[X]$, and thus Noetherian.
\end{proof}

Recall that a $\BB^{\otimes E}_{\mid K}$-pair is a pair $W=(W_e,W_{\dr}^+)$ where $W_e$ is a free $\bee$-module of finite rank endowed with a continuous semi-linear action of $G_K$, and $W_{\dr}^+$ is a $G_K$-stable $E \otimes_{\Qp} \bdr^+$-lattice in $W_{\dr} = (E \otimes_{\Qp} \bdr) \otimes_{\bee} W_e$.

\begin{prop}
\label{drplat}
If $W_e$ is a $\bee$-representation of $G_K$, then $(E \otimes_{\Qp} \bdr) \otimes_{\bee} W_e$ admits an $E \otimes_{\Qp} \bdr^+$-lattice stable under $G_K$.
\end{prop}

\begin{proof}
See \S 3.5 of \cite{F04}. The same argument gives an $E \otimes_{\Qp} \bdr^+$-lattice instead of a $\bdr^+$-lattice if one starts from an $E \otimes_{\Qp} \bdr$-representation.
\end{proof}

Recall that Nakamura has classified the $\BB^{\otimes E}_{\mid K}$-pairs of rank $1$, under the assumption that $E$ contains the Galois closure of $K$. Given a character $\delta : K^\times \to E^\times$, he constructs in \S 1.4 of \cite{KN09} a rank $1$ $\BB^{\otimes E}_{\mid K}$-pair $W(\delta)$, that we denote by $\BB(\delta)$, and proves that every rank $1$ $\BB^{\otimes E}_{\mid K}$-pair is of this form for a unique $\delta$. We have $\BB(\delta_1) \otimes \BB(\delta_2) = \BB(\delta_1 \delta_2)$ (\S 1.4 of \cite{KN09}). We denote by $\BB(\delta)_e$ the $\bee$-component of $\BB(\delta)$.

Recall (see for instance \S 2 of \cite{BC10} or \S 1.3 of \cite{KN09}) that $\BB^{\otimes E}_{\mid K}$-pairs have slopes. This comes from the equivalence of categories between $\BB^{\otimes E}_{\mid K}$-pairs and $(\phi,\Gamma)$-modules over the Robba ring, and Kedlaya's constructions and results for $\phi$-modules over the Robba ring (see \cite{KK04}). In particular, one can define the notion of isoclinic (pure of a certain slope) $\BB^{\otimes E}_{\mid K}$-pairs. For example, if $V$ is an $E$-linear representation of $G_K$, then $W(V) = (\bee \otimes_E V, (E \otimes_{\Qp} \bdr^+) \otimes_E  V)$ is pure of slope $0$, and every $\BB^{\otimes E}_{\mid K}$-pair that  is pure of slope $0$ is of this form (proposition 2.2 of \cite{BC10}).

We have the following slope filtration theorem (see theorem 2.1 of \cite{BC10}).

\begin{theo}
\label{slopefil}
If $W$ is a $\BB^{\otimes E}_{\mid K}$-pair,  there is a canonical filtration $\{0\} = W_0 \subset W_1 \subset \cdots \subset W_\ell = W$ by sub $\BB^{\otimes E}_{\mid K}$-pairs such that
\begin{enumerate}
\item for every $1 \leq i \leq \ell$, the quotient $W_i/W_{i-1}$ is isoclinic;
\item if $s_i$ is the slope of $W_i/W_{i-1}$, then $s_1 < s_2 < \cdots < s_\ell$.
\end{enumerate}
\end{theo}

The following proposition gathers the results that we need concerning slopes of $\BB^{\otimes E}_{\mid K}$-pairs. Recall that $\Hom(X,Y) = (\Hom_{E \otimes_{\Qp} \be}(X_e,Y_e), \Hom_{E \otimes_{\Qp} \bdr^+}(X_{\dr}^+,Y_{\dr}^+))$.

\begin{prop}
\label{sloprop}
If $X$ is pure of slope $s$ and $Y$ is pure of slope $t$, then 
\begin{enumerate}
\item $\Hom(X,Y)$ is pure of slope $t-s$ and $X \otimes Y$ is pure of slope $s+t$;
\item if $X$ and $Y$ have the same rank and $X \subset Y$ and $s=t$, then $X=Y$;
\item if $Y$ is a direct summand of $X$, then $s=t$.
\end{enumerate}
\end{prop}

\begin{proof}
For (1), see theorem 6.10 and proposition 5.13 of \cite{KK04}. For (2), we can take determinants and assume that $X$ and $Y$ are of rank $1$. The claim is then proposition 2.3 of \cite{LB08}. Item (3) follows from the fact that if $X=Y\oplus Z$, then the set of slopes of $X$ is the union of those of $Y$ and $Z$ (proposition 5.13 of \cite{KK04}).
\end{proof}

\section{The ring $\bee$}

Recall that $\bee = E \otimes_{\Qp} \be$. In this section, we determine the units of $\bee$ and study the rank $1$ $\bee$-representations of $G_E$. Let $q=p^h$ be the cardinality of the residue field of $\OO_E$, so that $E_0 = \QQ_{p^h}$. Let $\phi_E : E \otimes_{E_0} \bcris \to E \otimes_{E_0} \bcris$ be the map $\Id \otimes \phi^h$.

\begin{prop}
\label{fes}
We have an exact sequence
\[ 0 \to E \to \bee \to (E \otimes_{\Qp} \bdr) / (E \otimes_{\Qp} \bdr^+) \to 0. \]
\end{prop}

\begin{proof}
This follows from tensoring by $E$ the usual fundamental exact sequence $0 \to \Qp \to \be \to \bdr / \bdr^+ \to 0$ (proposition 1.17 of \cite{BK90}).
\end{proof}

\begin{prop}
\label{beealt}
The natural map $\bee \to (E \otimes_{E_0} \bcris)^{\phi_E=1}$ is an isomorphism.
\end{prop}

\begin{proof}
Since $\phi_E$ is $E$-linear, we have $(E \otimes_{E_0} \bcris)^{\phi_E=1} = E \otimes_{E_0} \bcris^{\phi^h =1}$ and it is therefore enough to prove that $\bcris^{\phi^h =1} = \QQ_{p^h} \otimes_{\Qp} \bcris^{\phi=1}$. The group $\Gal(\QQ_{p^h}/\Qp)$ acts $\QQ_{p^h}$-semi-linearly on $\bcris^{\phi^h =1}$ via $\phi$, and the claim follows from Galois descent (Speiser's lemma).
\end{proof}

\begin{rema}
\label{genogqp}
The isomorphism of proposition \ref{beealt} is $G_E$-equivariant. 

In addition, if $g \in G_{\Qp}$ acts by $\Id \otimes g$ on $E \otimes_{\Qp} \be$, then it acts by $\Id \otimes g \phi^{-n(g)}$ on $(E \otimes_{E_0} \bcris)^{\phi_E=1}$  (where $n(g)$ is defined below).
\end{rema}

Let $\pi$ be a uniformizer of $\OO_E$, and let $\chi_\pi$ denote the Lubin-Tate character $\chi_\pi : G_E \to \OO_E^\times$ attached to $\pi$. For each $\tau \in \Emb = \Gal(E/\Qp)$, let $n(\tau)$ be the element of  $\{0,\hdots,h-1\}$ such that $\tau=\phi^{n(\tau)}$ on $E_0$. Let $t_\tau \in E \otimes_{E_0} \bcris^+$ denote the element constructed in \S 5 of \cite{LB16}, where (in the notation of \cite{LB16}) we take $F=E$. We have $t_{\tau} = (\tau \otimes \phi^{n(\tau)})(t_{\Id})$. The element $t_{\Id}$ is also denoted by $t_\pi$ in \cite{LB16}, and it is the same as the element $t_E$ constructed in \S 9 of \cite{PC02}. The usual $t$ of $p$-adic Hodge theory is $t=t_{\Qp}$ for $\pi=p$.

For each $\sigma \in \Emb$, we have a map $E \otimes_{E_0}  \bcris^+ \to \bdr^+$ given by $x \mapsto (\sigma \otimes \phi^{n(\sigma)})(x)$, followed by the natural injection of $E \otimes_{E_0}  \bcris^+$ in $\bdr^+$ (theorem 4.2.4 of \cite{F94per}). Finally, note that $E \cdot \Qpnrhat = E \otimes_{E_0} \Qpnrhat$ is contained in $E \otimes_{E_0}  \bcris^+$.

\begin{prop}
\label{ttauprop}
Let the notation be as above.
\begin{enumerate}
\item We have $\phi_E(t_\tau) = \tau(\pi) \cdot t_\tau$ and $g(t_\tau) = \tau(\chi_\pi(g)) \cdot t_\tau$ if $g \in G_E$;
\item the $t$-adic valuation of the $\sigma$-component of the image of $t_\tau$ via the map $E \otimes_{E_0}  \bcris^+ \to E \otimes_{\Qp} \bdr = \prod_{\sigma \in \Emb} \bdr$ given by $x \mapsto \{ (\sigma \otimes \phi^{n(\sigma)})(x)\}_{\sigma \in \Emb}$ is $1$ if $\sigma=\tau^{-1}$ and $0$ otherwise;
\item there exists $u \in (E \cdot \Qpnrhat)^\times$ such that $\prod_{\tau \in \Emb} t_\tau = u \cdot t$ in $E \otimes_{E_0} \bcris$.
\end{enumerate}
\end{prop}

\begin{proof}
Since $t_{\tau} = (\tau \otimes \phi^{n(\tau)})(t_{\Id})$, it is enough to check (1) for $\tau=\Id$. The corresponding statement is at the end of \S 3 of \cite{LB16} (page 3578). Likewise, (2) follows from the case $\tau=\Id$. That case now follows from (1) and the fact that the Hodge-Tate weight of $\chi_\pi$ is $1$ at $\sigma=\Id$ and $0$ at $\sigma\neq\Id$. Finally, we have $\Nm_{E/\Qp}(\chi_\pi) = \chi_{\mathrm{cyc}} \eta$ where $\eta : G_E \to \Qp^\times$ is unramified, and by (1), this implies (3).
\end{proof}

Note that $t_\tau^{-1} \in E \otimes_{E_0} \bcris$ since $t_\tau$ divides $t$ in $\bcris^+$ by (3) of proposition \ref{ttauprop}.

\begin{prop}
\label{beunit}
If $\nbf = \{n_\tau\}_{\tau \in \Emb}$ is a tuple of integers whose sum is $0$, then there exists $u_{\nbf} \in (E \cdot \Qpnrhat)^\times$ such that $u = \prod_{\tau \in \Emb} t_{\tau}^{n_\tau} u_{\nbf}$ belongs to $\bee$. The element $u$ is then a unit of $\bee$ and every unit of $\bee$ is of this form up to multiplication by $E^\times$.
\end{prop}

\begin{proof}
Let $w=\phi_E( \prod_{\tau \in \Emb} t_{\tau}^{n_\tau} ) / \prod_{\tau \in \Emb} t_{\tau}^{n_\tau} = \prod_{\tau \in \Emb} \tau(\pi)^{n_\tau}$ by (1) of proposition \ref{ttauprop}. Since $\sum_{\tau \in \Emb} n_\tau = 0$, we have $w \in \OO_E^\times$. There exists $u_{\nbf} \in (E \cdot \Qpnrhat)^\times$ such that $\phi_E(u_{\nbf})/u_{\nbf} = w^{-1}$, and then $u = \prod_{\tau \in \Emb} t_{\tau}^{n_\tau} u_{\nbf}$ belongs to $\bee$.  The inverse of $u$ is $\prod_{\tau \in \Emb} t_{\tau}^{-n_\tau} u_{\nbf}^{-1}$ which also belongs to $\bee$, so that $u \in \bee^\times$.

We now show that every $u \in \bee^\times$ is of this form. Let $n_\tau$ be the $t$-adic valuation in $\bdr$ of the $\tau^{-1}$-component $u_{\tau^{-1}} = (\tau^{-1} \otimes \Id)(u)$ of the image of $u \in E \otimes_{\Qp} \be$ in $E \otimes_{\Qp} \bdr = \prod_{\sigma \in \Emb} \bdr$. Note that $u_\sigma \in \bee^\times$ for all $\sigma \in \Emb$ and that $\prod_{\sigma \in \Emb} u_\sigma \in (\bee^\times)^{\Emb} = \be^\times$. We have $\be^\times = \Qp^\times$ by lemma 1.1.8 of \cite{LB08}, so that $\sum_{\tau \in \Emb} n_\tau = 0$. By (2) of proposition \ref{ttauprop}, the element $u \cdot \prod_{\tau \in \Emb} t_{\tau}^{-n_\tau} u_{\nbf}^{-1}$ belongs to $(E \otimes_{\Qp} \bdr^+) \cap \bee^\times$, and $(E \otimes_{\Qp} \bdr^+) \cap \bee^\times= E^\times$ by proposition \ref{fes}.
\end{proof}

Recall that an $E$-linear representation is crystalline or de Rham if the underlying $\Qp$-linear representation is crystalline or de Rham. We say that a character $\delta : G_E \to E^\times$ is $\bee$-admissible if there exists $y \in \bee \setminus \{0\}$ such that $\delta(g)=g(y)/y$. Such a character is then crystalline, hence also de Rham.

\begin{prop}
\label{begkstab}
If $y \in \bee \setminus \{0\}$ is such that $y \cdot \bee$ is stable under $G_E$, then $y \in \bee^\times$ and there exists $n_\tau \in \ZZ$ with $\sum _{\tau \in \Emb} n_{\tau} = 0$ and $y_0 \in (E \cdot \Qpnrhat)^\times$ such that $y = \prod_{\tau \in \Emb} t_{\tau}^{n_\tau} y_0$.
\end{prop}

\begin{proof}
If $y \cdot \bee$ is stable under $G_E$, then $g(y)/y \in \bee$ for all $g \in G_E$. Note that if $z \in \bdr^\times$, then $g(z)/z \in \bdr^+$. This implies that $g(y)/y \in \bee \cap (E \otimes_{\Qp} \bdr^+)$. By proposition \ref{fes}, $g(y)/y \in E^\times$. The map $\delta : G_E \to E^\times$ given by $\delta(g) = g(y)/y$ is a crystalline character of $G_E$, and hence of the form $\prod_{\tau \in \Emb} \tau(\chi_\pi)^{n_\tau} \eta_0$ where $n_\tau \in \ZZ$ and $\eta_0 : G_E \to E^\times$ is unramified. This implies that there exists $y_0 \in (E \cdot \Qpnrhat)^\times$ such that $y = \prod_{\tau \in \Emb} t_\tau^{n_\tau} y_0$. If $y \in \bee$, then $\phi_E(y) = y$ so that $\sum_{\tau \in \Emb} n_\tau = 0$ by (1) of proposition \ref{ttauprop}, and hence $y \in \bee^\times$.
\end{proof}

\begin{coro}
\label{beadm}
If $\delta : G_E \to E^\times$ is a $\bee$-admissible character, then $\delta$ is de Rham and the sum of its weights at all $\tau \in \Emb$ is $0$. Conversely, any character $\delta : G_E \to E^\times$ that is de Rham with the sum of its weights at all $\tau \in \Emb$ equal to $0$ is the product of a $\bee$-admissible character by a potentially unramified character.
\end{coro}

\begin{proof}
The first assertion follows immediately from proposition \ref{begkstab}. We now prove the second assertion. If $\delta : G_E \to E^\times$ is de Rham, it is of the form $\prod_{\tau \in \Emb} \tau(\chi_\pi)^{n_\tau} \eta_0$ where $n_\tau \in \ZZ$ and $\eta_0 : G_E \to E^\times$ is potentially unramified. Let $\nbf = \{n_\tau\}_{\tau \in \Emb}$ and $u$ be the corresponding unit (proposition \ref{beunit}). If $g \in G_E$, then $g(u)/u = \prod_{\tau \in \Emb} \tau(\chi_\pi(g))^{n_\tau} \eta_u(g)$ where $\eta_u : G_E \to E^\times$ is unramified. The second assertion then follows from this.
\end{proof}

A $\bee$-representation of $G_K$ is a free $\bee$-module of finite rank with a semi-linear and continuous action of $G_K$ (recall that $G_K$ acts linearly on $E$). If $\delta \in H^1(G_K,\bee^\times)$ (for example if $\delta : G_K \to E^\times$ is a character), we denote by $\bee(\delta)$ the resulting rank $1$ $\bee$-representation of $G_K$.

\begin{prop}
\label{suberep}
If $W_e$ is a $\bee$-representation of $G_K$, and if $X_e$ is a sub $\bee$-module of $W_e$ stable under $G_K$, then $X_e$ is a free $\bee$-module, and it is saturated in $W_e$.
\end{prop}

\begin{proof}
See lemma 1.10 of \cite{KN09}.
\end{proof}

\begin{prop}
\label{berkonebis}
If $W$ is a rank $1$ $\bee$-representation of $G_E$, then there exists $\delta : G_E \to E^\times$ such that $W=\bee(\delta)$.
\end{prop}

\begin{proof}
If we choose a basis $w$ of $W$, then $g(w)=\delta(g)w$ with $\delta(g) \in \bee^\times$, so that $\delta(g)$ is of the form $\prod_{\tau \in \Emb} t_\tau^{n_\tau(g)} u_{\nbf(g)}$ by proposition \ref{beunit}. Since $\delta(gh) = \delta(g) g(\delta(h))$, (1) of proposition \ref{ttauprop} implies that the maps $n_\tau : G_E \to \ZZ$ are continuous homomorphisms. They are therefore trivial, and this implies that $\delta(g) \in E^\times$.
\end{proof}

\begin{rema}
\label{notunik}
The character $\delta$ in proposition \ref{berkonebis} is not unique, since it can be multiplied by any $\bee$-admissible character of $G_E$.
\end{rema}

\begin{rema}
\label{notchar}
If $K \neq E$, it is not necessarily true that every rank $1$ $\bee$-representation of $G_K$ is of the form $\bee(\delta)$ for a character $\delta : G_K \to E^\times$.
\end{rema}

\begin{proof}
Take $E=\Qp(\sqrt{p})$ and $K=\Qp$ and $W= (E \otimes_{\Qp} \bcris)^{\phi=\pi} = t_{\Id} \cdot \bee$. The $E$-linear action of $G_{\Qp}$ on $W$ is given by the map $\delta : g \mapsto g(t_{\Id})/t_{\Id}$. If $g \in G_E$, then $\delta(g) = \chi_\pi(g)$. If $u=t_{\Id}^n t_{\tau}^{-n} u_{n,-n} \in \bee^{\times}$ as in proposition \ref{beunit}, and $g \notin G_E$, then $g(ut_{\Id})/ut_{\Id} = t_{\Id}^{-2n-1} t_\tau^{2n+1} v$ with $v \in (E \cdot \Qpnrhat)^\times$. Therefore, there is no character $\eta : G_{\Qp} \to E^\times$ such that $W = \bee(\eta)$. 

Note that $W$ is the $\bee$-component of the $\BB^{\otimes E}_{\mid K}$-pair $W_0^{-1}$ of \S 1.4 of \cite{KN09}.
\end{proof}

\begin{rema}
\label{ffpid}
The results of this section provide a new proof of proposition \ref{bepid}. 
\end{rema}

\begin{proof}
By theorem 6.5.2 of \cite{FF18}, the ring $(E \otimes_{E_0} \bcris^+[1/t_{\Id}])^{\phi_E=1}$ is a PID. Since we have shown $\bee$ is a localization of $(E \otimes_{E_0} \bcris^+[1/t_{\Id}])^{\phi_E=1}$, it is itself a PID.
\end{proof}

\begin{prop}
\label{fracbeinv}
We have $\Frac(\bee)^{G_K}=E$.
\end{prop}

\begin{proof}
Take $x/y \in \Frac(\bee)^{G_K}$ with $x$, $y \in \bee$ coprime. If $g \in G_K$, then $g(x)y=xg(y)$ so that $x$ divides $g(x)$ and $y$ divides $g(y)$ in $\bee$ (recall that $\bee$ is a PID). By proposition \ref{begkstab}, $x$ and $y$ belong to $\bee^\times$. This implies that $x/y \in \bee^{G_K}=E$.
\end{proof}

\begin{coro}
\label{wegkinv}
If $W_e$ is a $\bee$-representation of $G_K$, then $\dim_E W_e^{G_K} \leq \operatorname{rk} W_e$.
\end{coro}

\begin{proof}
By a standard argument, proposition \ref{fracbeinv} implies that the map $\bee \otimes_E W_e^{G_K} \to W_e$ is injective. This implies the corollary.
\end{proof}

\section{Triangulable representations}

In this section, we study triangulable $\BB^{\otimes E}_{\mid K}$-pairs and $\bee$-representations of $G_K$. We say  that a $\BB^{\otimes E}_{\mid K}$-pair is irreducible if it has no non-trivial saturated sub $\BB^{\otimes E}_{\mid K}$-pair (see \S 2.1 of \cite{LB08}).

\begin{prop}
\label{weirred}
If $W = (W_e,W_{\dr}^+)$ is an irreducible $\BB^{\otimes E}_{\mid K}$-pair, then $W_e$ is an irreducible $\bee$-representation of $G_K$.
\end{prop}

\begin{proof}
Let $X_e$ be a sub-object of $W_e$. By proposition \ref{suberep}, it is a saturated and free submodule of $W_e$. The space $X_{\dr}^+ = X_{\dr} \cap W_{\dr}^+$ is an $E \otimes_{\Qp} \bdr^+$ lattice of $X_{\dr}$ stable under $G_K$. Hence $X=(X_e,X_{\dr}^+)$ is a saturated sub $\BB^{\otimes E}_{\mid K}$-pair of $W$.
\end{proof}

\begin{coro}
\label{tribe}
If $W$ is a $\BB^{\otimes E}_{\mid K}$-pair, then $W$ is split triangulable as a $\BB^{\otimes E}_{\mid K}$-pair if and only if $W_e$ is split triangulable as a $\bee$-representation of $G_K$.
\end{coro}

\begin{proof}
It is clear that if $W$ is split triangulable, then so is $W_e$. Conversely, the proof of proposition \ref{weirred} shows how to construct a triangulation of $W$ from a triangulation of $W_e$.
\end{proof}

Let $\Delta$ be a set of rank $1$ semi-linear $\bee$-representations of $G_K$. Recall that a $\BB^{\otimes E}_{\mid K}$-pair is split $\Delta$-triangulable if it is split triangulable, and the rank $1$ $\bee$-representations of $G_K$ that come from the triangulation are all in $\Delta$.

\begin{prop}
\label{extriang}
If $0 \to W' \to W \to W'' \to 0$ is an exact sequence of $\BB^{\otimes E}_{\mid K}$-pairs, then $W$ is split $\Delta$-triangulable if and only if $W'$ and $W''$ are split $\Delta$-triangulable.
\end{prop}

\begin{proof}
If $W'$ and $W''$ are split $\Delta$-triangulable, then $W$ is obviously split $\Delta$-triangulable. We now prove the converse. If $W_e$ admits a triangulation, then so do $W_e'$ and $W_e''$. By corollary \ref{tribe}, $W'$ and $W''$ are therefore split triangulable. Proposition \ref{suberep} implies that two different triangulations of $W_e$ give rise to two composition series of $W_e$ (seen as a $\bee$-representation of $G_K$). The set of rank $1$ $\bee$-representations attached to any triangulation of $W_e$ is therefore well-defined up to permutation by the Jordan-H\"older theorem. Hence if $W$ is split $\Delta$-triangulable, then so are $W'$ and $W''$.
\end{proof}

\begin{prop}
\label{splitbe}
If $W_e$ is an irreducible $\bee$-representation of $G_K$, and $\delta \in H^1(G_K,\bee^\times)$, then every surjective map $\pi : \End(W_e) \to \bee(\delta)$ of $\bee$-representations of $G_K$ is split.
\end{prop}

\begin{proof}
Write $\bee(\delta) = \bee \cdot e_\delta$, where $g(e_\delta) = \delta(g) e_\delta$ with $\delta(g) \in \bee^\times$. Recall that if $A$ is a ring and $M$ is a free $A$-module, then $\End_A(M)$ is its own dual, for the pairing $(f,g) \mapsto \Tr(fg)$. The map $\pi$ is therefore of the form $f \mapsto \Tr(fh) \cdot e_\delta$ for some $h \in \End(W_e)$. The map $h$ satisfies $g(h) = \delta(g)^{-1} h$, and therefore gives rise to a $G_K$-equivariant map $h : W_e \to W_e(\delta)$. Since $W_e$ is irreducible, $h$ is invertible. We can then write $\End(W_e) = \ker(\pi) \oplus \bee \cdot h^{-1}$, which shows that $\pi$ is split. 
\end{proof}

\begin{theo}
\label{bedirsum}
If $W_e$ is an irreducible $\bee$-representation of $G_K$ such that $\End(W_e)$ is split triangulable, then the triangulation of $\End(W_e)$ splits.
\end{theo}

\begin{proof}
Write $\{0\} = X_0 \subset X_1 \subset \cdots \subset X_d = \End(W_e)$, and $X_i/X_{i-1} = \bee(\delta_i)$ for some $\delta_i \in H^1(G_K,\bee^\times)$. By proposition \ref{splitbe}, the exact sequence $0 \to X_{d-1} \to \End(W_e) \to \bee(\delta_d) \to 0$ is split, and therefore $\End(W_e) = X_{d-1} \oplus \bee(\delta_d)$.

Suppose that we have an isomorphism $\End(W_e) = X_j \oplus \bee(\delta_{j+1}) \oplus \cdots \oplus \bee(\delta_d)$. Let $\pi_j$ denote the composition $\End(W_e) \to X_j \to \bee(\delta_j)$. By proposition \ref{splitbe},  $\End(W_e) = \ker(\pi_j) \oplus \bee(\delta_j)$. We have $\ker(\pi_j) = X_{j-1} \oplus \bee(\delta_{j+1}) \oplus \cdots \oplus \bee(\delta_d)$, so that $\End(W_e) = X_{j-1} \oplus \bee(\delta_j) \oplus \cdots \oplus \bee(\delta_d)$. The claim  follows by induction.
\end{proof}

\begin{rema}
\label{deligne}
Theorem \ref{bedirsum} is reminiscent of the following result of Chevalley: if $G$ is any group and if $X$ and $Y$ are finite dimensional semi-simple characteristic $0$ representations of $G$, then $X \otimes Y$ is also semi-simple. The same holds for semi-linear representations and, more generally, in any Tannakian category over a field of characteristic  $0$ \cite{PD16}.
\end{rema}

\section{Cohomology of $B$-pairs} 

The cohomology of $\BB^{\otimes E}_{\mid K}$-pairs is defined and studied in \S 2.1 of \cite{KN09}. We recall what we need. Let $W$ be a $\BB^{\otimes E}_{\mid K}$-pair. Nakamura constructs an $E$-vector space $H^1(G_K,W)$ that has the following properties
\begin{enumerate}
\item $H^1(G_K,W) = \Ext^1(\BB,W)$ (i.e. it classifies the extensions of $\BB^{\otimes E}_{\mid K}$-pairs);
\item there is an exact sequence of $E$-vector spaces
\[  W_{\dr}^{G_K} \to H^1(G_K,W) \to H^1(G_K,W_e) \oplus H^1(G_K,W_{\dr}^+). \]
\end{enumerate}

If $W$ is a rank $1$ $\BB^{\otimes E}_{\mid K}$-pair with $W_e \in \Delta(\Qp)$, then $W_{\dr}^{G_{\Qp}}$ is an $E$-vector space of dimension $1$ or $0$, depending on whether $W_e$ (extended to $G_{\Qp}$) is de Rham or not. Since $W_{\dr}^{G_K} = K \otimes_{\Qp} W_{\dr}^{G_{\Qp}}$, this implies that $W_{\dr}^{G_K} = \{0\}$ if $W$ is not de Rham. Note that if $W$ is a rank $1$ $\BB^{\otimes E}_{\mid K}$-pair with $K \neq \Qp$, then $W$ may be ``partially de Rham'' in the sense of \cite{YD17}, so that in general $W_{\dr}^{G_K}$ can be non-zero even if $W$ is not de Rham.

\begin{prop}
\label{drinj}
If $W_{\dr}^+$ is a free $E \otimes_{\Qp} \bdr^+$-representation of $G_K$ of rank $1$, the map $H^1(G_K,W_{\dr}^+) \to H^1(G_K,W_{\dr})$ is injective.
\end{prop}

\begin{proof}
Since we have an exact sequence
\[ W_{\dr}^{G_K} \to (W_{\dr}/W_{\dr}^+)^{G_K} \to H^1(G_K,W_{\dr}^+) \to H^1(G_K,W_{\dr}), \]
it is enough to show that $W_{\dr}^{G_K} \to (W_{\dr}/W_{\dr}^+)^{G_K}$ is surjective. To prove this, we can replace $K$ by a finite extension $L$, and in particular we can assume that $L$ contains $E$. In this case, $W_{\dr}^+ {\mid}_{G_L}$ is a direct sum of rank $1$ $\bdr^+$-representations of $G_L$. 

Let $X_{\dr}^+$ be a rank $1$ $\bdr^+$-representation of $G_L$. The $L$-vector space $X_{\dr}^{G_L}$ is of dimension $0$ or $1$. If $\dim_L X_{\dr}^{G_L} = 1$, then $X_{\dr}$ is de Rham, and the map $X_{\dr}^{G_L} \to (X_{\dr}/X_{\dr}^+)^{G_L}$ is surjective by the same argument as in lemma 3.8.1 of \cite{BK90} (see lemma 2.6 of \cite{KN09}). If $\dim_L X_{\dr}^{G_L} = 0$, then for every $i \in \ZZ$, we have $(t^i X_{\dr}^+ / t^{i+1} X_{\dr}^+)^{G_L} = 0$ by proposition 3.21 of \cite{F04}. This implies that $(X_{\dr} / X_{\dr}^+)^{G_L} = 0$, so that the map $X_{\dr}^{G_L} \to (X_{\dr}/X_{\dr}^+)^{G_L}$ is also surjective.
\end{proof}

\begin{coro}
\label{dirsum}
If $X$ is a direct sum of rank $1$ $\BB^{\otimes E}_{\mid K}$-pairs, the map $H^1(G_K,X_{\dr}^+) \to H^1(G_K,X_{\dr})$ is injective.
\end{coro}

Recall that every rank $1$ $\BB^{\otimes E}_{\mid K}$-pair is of the form $\BB(\delta)$ for a unique $\delta : K^\times \to E^\times$.

\begin{prop}
\label{bpairsplit}
If a $\BB^{\otimes E}_{\mid K}$-pair $W$ is split $\Delta(\Qp)$-triangulable, with subquotients $\{ \BB(\delta_i) \}_i$ such that $\BB(\delta_i \delta_j^{-1})$ is not de Rham for any $i \neq j$, and if the corresponding triangulation of $W_e$ splits as a direct sum of $1$-dimensional $\bee$-representations, then the triangulation of $W$ splits.
\end{prop}

\begin{proof}
Let $0 = W_0 \subset W_1 \subset \cdots \subset W_d = W$ be the given triangulation of $W$. We prove by induction on $j$ that $W_j= \BB(\delta_1) \oplus \cdots \oplus \BB(\delta_j)$. This is true for $j=1$, assume it holds for $j-1$. Write $0 \to W_{j-1} \to W_j \to \BB(\delta_j) \to 0$ and $W_{j-1} = \BB(\delta_1) \oplus \cdots \oplus \BB(\delta_{j-1})$. Let $X=W_{j-1} (\delta_j^{-1})$ and $Y = W_j (\delta_j^{-1})$. The $\BB^{\otimes E}_{\mid K}$-pair $Y$ corresponds to a class in $H^1(G_K,X)$. The $\bee$-representation $Y_e$ is split, and therefore so is $Y_{\dr}$. By corollary \ref{dirsum}, so is $Y_{\dr}^+$. The class of $Y$ in $H^1(G_K,X)$ is therefore in the kernel of $H^1(G_K,X) \to H^1(G_K,X_e) \oplus H^1(G_K,X_{\dr}^+)$. Since $X_{\dr}^{G_K}=0$ by hypothesis, Nakamura's exact sequence (2) above implies that the class of $Y$ is trivial and hence $W_j = W_{j-1} \oplus \BB(\delta_j)$. The proposition follows by induction. 
\end{proof}

\section{Proof of the main theorem}

In this section, we prove theorem B. Let $F$ be a finite extension of $E$ of degree $\geq 2$, and write $F \otimes_E F = \oplus_i F_i$. There are at least two summands since $F$ itself is one of them.

\begin{prop}
\label{extfe}
Let $F/E$ be as above, and let $W$ be an $F$-linear representation of $G_K$. We have $F \otimes_E W = \oplus_i (F_i \otimes_F W)$ as $F$-linear representations of $G_K$.
\end{prop}

\begin{proof}
We have $F \otimes_E W = (F \otimes_E F) \otimes_F W = \oplus_i (F_i \otimes_F W)$.
\end{proof}

\begin{coro}
\label{notabsirr}
If $W$ is a $\bee$-representation of $G_K$ that has an $F$-linear structure, then $W$ becomes reducible after extending scalars from $E$ to $F$.
\end{coro}

Let us say that a $\BB^{\otimes E}_{\mid K}$-pair $W$ is completely irreducible if $(F \otimes_E W)|_{G_L}$ is an irreducible $\BB^{\otimes F}_{\mid L}$-pair for all finite extensions $F$ of $E$ and $L$ of $K$.

\begin{prop}
\label{tensirr}
If $K=E$ and if $X$ and $Y$ are two completely irreducible $\BB^{\otimes E}_{\mid K}$-pairs such that $\Hom(X,Y)$ is split $\Delta(\Qp)$-triangulable, then $X$ and $Y$ are of rank $1$.
\end{prop}

\begin{proof}
Let $\{ \BB(\delta_i) \}_i$ be the rank $1$ subquotients of the triangulation of $\Hom(X,Y)$. We have an inclusion $\BB(\delta_1) \subset \Hom(X,Y)$. This gives rise to a non-zero map $X \to Y(\delta_1^{-1})$ of $\BB^{\otimes E}_{\mid K}$-pairs. Write $\BB(\delta_1)_e = \bee(\mu_1)$ for some $\mu_1 : G_K \to \bee^\times$ (recall that $K=E$). Since $X$ and $Y$ are irreducible, $X_e$ and $Y_e$ are irreducible $\bee$-representations of $G_K$ (proposition \ref{weirred}), and the map $X_e \to Y_e(\mu_1^{-1})$ is therefore an isomorphism. This implies that $\Hom(X_e,Y_e) = \End(X_e)(\mu_1)$, so that $\End(X_e)$ is split triangulable. By theorem \ref{bedirsum}, the triangulation of $\End(X_e)$ splits. The triangulation of $\Hom(X_e,Y_e) = \End(X_e)(\mu_1)$ therefore also splits. Let $n$ be the common rank of $X$ and $Y$.

Suppose that none of the $\BB(\delta_i \delta_j^{-1})$ are de Rham for any $i \neq j$. By proposition \ref{bpairsplit} applied to $W=\Hom(X,Y)$, the triangulation of $\Hom(X,Y)$ splits. We can therefore write $\Hom(X,Y) = \oplus_i \BB(\delta_i)$. Since $X$ and $Y$ are both irreducible, they are pure of some slopes $s$ and $t$ by theorem \ref{slopefil}. The $\BB^{\otimes E}_{\mid K}$-pair $\Hom(X,Y)$ is then pure of slope $t-s$ by (1) of proposition \ref{sloprop}. By (3) of ibid, each of the $\BB(\delta_i)$ is also pure of slope $t-s$. Each $\BB(\delta_i)$ gives rise to a map $X \to Y(\delta_i^{-1})$, which is an isomorphism of $\BB^{\otimes E}_{\mid K}$-pairs by (2) of ibid, since $X$ and $Y(\delta_i^{-1})$ are both pure of slope $s$. By taking determinants, we get $\delta_i^n = \det(Y)\det(X)^{-1}$ for every $i$. This implies that $(\delta_i \delta_j^{-1})^n=1$ so that $\delta_i \delta_j^{-1}$ is of finite order, and $\BB(\delta_i \delta_j^{-1})$ is de Rham (lemma 4.1 of \cite{KN09}), contradicting our assumption.

Therefore, one of the $\BB(\delta_i \delta_j^{-1})$ is de Rham for some $i \neq j$. Write $\BB(\delta_k)_e = \bee(\mu_k)$ where the $\mu_k$ are characters $G_K \to E^\times$ (recall that $K=E$), so that $\End(X_e)(\mu_1) = \oplus_k \bee(\mu_k)$ as $\bee$-representations of $G_K$. The fact that $\BB(\delta_i \delta_j^{-1})$ is de Rham implies that $\mu_i \mu_j^{-1}$ is de Rham. We then have $X_e = X_e(\mu_1 \mu_i^{-1}) = X_e(\mu_1 \mu_j^{-1})$, so that $X_e=X_e(\mu_i \mu_j^{-1})$. By taking determinants, we find  that $\BB_{e,E}((\mu_i \mu_j^{-1})^n) = \BB_{e,E}$ and therefore by corollary \ref{beadm}, $(\mu_i \mu_j^{-1})^n : G_K \to E^\times$ is de Rham and the sum of its weights is $0$. This implies that the sum of the weights of $\mu_i \mu_j^{-1} : G_K \to E^\times$ is $0$. By corollary \ref{beadm}, $\mu_i \mu_j^{-1} = \chi \eta$ with $\chi : G_K \to E^\times$ a $\BB_{e,E}$-admissible character and $\eta : G_K \to E^\times$ potentially unramified. Since $X_e(\chi \eta)=X_e$ and $X_e(\chi)=X_e$, we get $X_e(\eta) = X_e$. By taking determinants, we get that $\eta^n$ is $\BB_{e,E}$-admissible. Since $\eta^n$ is also potentially unramified, and $\bee \cap (\Qpbar \cdot \Qpnrhat) = E$, it is trivial. Hence $\eta$ is a character of finite order of $G_K$, and so there exists a finite extension $L$ of $K$ such that $\mu_i=\chi \mu_j$ on $G_L$.

The space $\End(X_e)(\mu_1)$ contains $\bee(\mu_j) \oplus \bee(\mu_i)$, which is isomorphic to $\bee(\mu_j) \oplus \bee(\mu_j)$ after restricting to $G_L$. Let $f$ and $g$ be the two resulting isomorphisms $X_e \to X_e(\mu_1 \mu_j^{-1})$. The map $h = f^{-1} \circ g : X_e \to X_e$ is $G_L$-equivariant and is not in $E^\times \cdot \Id$ since $f$ and $g$ are $\bee$-linearly independent. Therefore, $\End(X_e)^{G_L}$ is strictly larger than $E$. 

Since $X_e|_{G_L}$ is irreducible, Schur's lemma and corollary \ref{wegkinv} imply that $\End(X_e)^{G_L}$ contains a field $F$ such that $[F:E] \geq 2$ (for example, $F=E[h]$). Hence $X_e |_{G_L}$ has an $F$-linear structure. Corollary \ref{notabsirr} implies that $(F \otimes_E X_e) |_{G_L}$ is reducible. By proposition \ref{weirred}, $X$ is not completely irreducible. This is a contradiction, so $X$ had to be of rank $1$. Since $X$ and $Y$ have the same rank, we are done.
\end{proof}

We now recall and prove theorem B. A strict sub-quotient of a $\BB^{\otimes E}_{\mid K}$-pair is a quotient of a saturated sub $\BB^{\otimes E}_{\mid K}$-pair.

\begin{theo}
\label{maintheo}
If $X$ and $Y$ are two non-zero $\BB^{\otimes E}_{\mid K}$-pairs whose tensor product is $\Delta(\Qp)$-triangulable, then $X$ and $Y$ are both potentially triangulable.
\end{theo}

\begin{proof}
We can replace $E$ and $K$ by finite extensions $F$ and $L$ if necessary, and write $X$ and $Y$ as successive extensions of completely irreducible $\BB^{\otimes F}_{\mid L}$-pairs with $F=L$. If $X'$ and $Y'$ are two strict sub-quotients of $X$ and $Y$, then $X' \otimes Y'$ is a strict sub-quotient of $X \otimes Y$, and it is $\Delta(\Qp)$-triangulable by proposition \ref{extriang}. Proposition \ref{tensirr}, applied to $(X')^*$ and $Y'$ so that $X' \otimes Y' = \Hom((X')^*,Y')$, tells us that $X'$ and $Y'$ are of rank $1$. 

Hence the $\BB^{\otimes F}_{\mid L}$-pairs $(F \otimes_E X)|_{G_L}$ and $(F \otimes_E Y)|_{G_L}$ are split triangulable.
\end{proof}

\begin{coro}
\label{betenstri}
If $X_e$ and $Y_e$ are two $\bee$-representations of $G_K$ whose tensor product is triangulable, with the rank $1$ sub-quotients extending to $\bee$-representations of $G_{\Qp}$, then $X_e$ and $Y_e$ are both potentially triangulable.
\end{coro}

\begin{proof}
By proposition \ref{drplat}, $X_e$ and $Y_e$ extend to $\BB^{\otimes E}_{\mid K}$-pairs.  The result follows from corollary \ref{tribe} and theorem \ref{maintheo}.
\end{proof}

We finish with an example of a representation $V$ such that $V \otimes_E V$ is trianguline, but $V$ itself is not trianguline. This shows that the ``potentially'' in the statement of theorem A cannot be avoided. Let $Q_8$ denote the quaternion group. If $p \equiv 3 \bmod{4}$, there is a Galois extension $K/\Qp$ such that $\Gal(K/\Qp)=Q_8$ (see II.3.6 of \cite{JY88}). Choose such a $p$ and $K$, and let $E$ be a finite extension of $\Qp$ containing $\sqrt{-1}$. The group $Q_8$ has a (unique) irreducible $2$-dimensional $E$-linear representation, which we inflate to a representation $V$ of $G_{\Qp}$. One can check that $V\otimes_E V$ is a direct sum of characters, hence trianguline, and that the semi-linear representation $\Frac(\bee) \otimes_E V$ is irreducible. This holds for all $E$ as above, so that $V$ is not trianguline.

\vspace{10pt}

{\noindent\bf Acknowledgements:} We thank L\'eo Poyeton and Sandra Rozensztajn for their comments, and David Hansen and Andrea Conti for their questions and discussions.

\providecommand{\bysame}{\leavevmode\hbox to3em{\hrulefill}\thinspace}
\providecommand{\MR}{\relax\ifhmode\unskip\space\fi MR }
\providecommand{\MRhref}[2]{%
  \href{http://www.ams.org/mathscinet-getitem?mr=#1}{#2}}
\providecommand{\href}[2]{#2}

\end{document}